\documentclass{article} 
\usepackage[utf8]{inputenc}
\usepackage{mathtools}
\usepackage{amstext}
\usepackage{amsfonts}
\usepackage{amssymb}
\usepackage{amsthm}
\usepackage{amsmath}
\usepackage{url}
\usepackage{enumitem}
\usepackage{cases}
\usepackage{verbatim}

\usepackage[backend=biber]{biblatex}
\newcommand{\mn}[0]{\medskip\noindent}
\newcommand{\pr}[1][]{\mathbb{P}}
\newcommand{\sm}[0]{\setminus}

\def\E{{\mathbb E}}
\def\P{{\mathbb P}}
\def\d{{\mathrm d}}
\newtheorem{thm}{Theorem}[section]

\newtheorem{cor}[thm]{Corollary}

\newtheorem{prop}[thm]{Proposition}
\newtheorem{conj}{Conjecture}

\newtheorem{case}{Case}

\title{Structure of the largest subgraphs of $G_{n,p}$ with a given matching number}
\author{Abigail Raz\thanks{Department of Mathematics, Rutgers University, Piscataway NJ. Email: ajr224@math.rutgers.edu}}
\date{}
\begin{document}
\maketitle
\begin{abstract}
This paper examines the structure of the largest subgraphs of the Erd\H{o}s-R{\'e}nyi random graph, $G_{n,p}$, with a given matching number. This extends a result of Erd\H{o}s and Gallai who, in 1959, gave a classification of the structures of the largest subgraphs of $K_n$ with a given matching number. We show that their result extends to $G_{n,p}$ with high probability when $p\ge \frac{8 \ln n}{n}$ or $p \ll \frac{1}{n}$, but that it does not extend (again with high probability) when $\frac{4\ln(2e)}{n} < p< \frac{\ln n}{3n}$. 
\end{abstract}
\section{Introduction}
Recall that for a graph $G$ the matching number (size of a largest matching) is denoted $\nu(G)$. In what follows the \emph{size} of a graph is the number of edges. In 1959 Erd{\H o}s and Gallai \cite{EG} proved the following theorem on the size of the largest subgraphs of $K_n$ with a given matching number:
\begin{thm}\cite[Theorem 4.1]{EG}\label{EG}
Each largest subgraphs of $K_n$ with matching number $k$ has one of the following forms.
\begin{enumerate}[label=(\alph*)]
\item all edges within a fixed set of vertices of size $2k+1 $;
\item all edges meeting a fixed set of vertices of size $k$. 
\end{enumerate}
\end{thm}
Erd{\H o}s conjectured that this result can be extended from graphs to $l$-uniform hypergraphs for all $l$. 
\begin{conj}(Erd{\H o}s' Matching Conjecture)
The largest subhypergraphs of $\mathcal{K}= {[n] \choose l}$ with matching number $k$ have size $$\max \left\{ \binom{l(k+1)-1}{l}, \binom{n}{l}- \binom{n-k}{l}\right\}.$$
\end{conj}
Note that these bounds are achieved by hypergraphs of the following forms:
\begin{enumerate}[label=(\alph*)]
\item all hyperedges within a fixed set of vertices of size $l(k+1)-1 $;
\item all hyperedges meeting a fixed set of vertices of size $k$.
\end{enumerate}
The case $l=2$ is Theorem \ref{EG}. The conjecture has also been proved for $l=3$ \cite{F1,F2,LM}, and when $k$ is not too close to $n/l$ \cite{F1,HLS}.
Note that as $k$ changes the optimal configuration shifts between the two forms. 

Here we show that Theorem \ref{EG} extends to $G_{n,p}$ for most values of $p=p(n)$. Let us say a graph $G$ has the \emph{EG Property} if for each $k\le \nu(G)$ every largest subgraph of $G$ with matching number $k$ has one of the two forms above, which we repeat for reference:
\begin{align}
    &\text{all edges within a fixed set of vertices of size }2k+1;  \label{form1}\\
    &\text{all edges meeting a fixed set of vertices of size }k. \label{form2}
\end{align}

\begin{thm} \label{main}
If $p \ge \frac{8\log n}{n}$, then with high probability\footnote{With high probability (``w.h.p.") means with probability tending to 1 as $n \rightarrow \infty$} $G_{n,p}$ has the EG Property. 
\end{thm}

The rest of this paper is organized as follows. Section \ref{prelim} covers basic definitions and terminology, and establishes edge density preliminaries. In Section~\ref{mainpf} we prove Theorem \ref{main}. We conclude in Section \ref{final} with some discussion of what happens when $p$ is not in the range covered by Theorem \ref{main}.

\section{Preliminaries} \label{prelim}
We use $N_A(x)$ for the set of neighbors of $x$ in $A$, and $d_A(x) = |N_A(x)|$, the degree of $x$ in $A$. Also, for $V= V(G)$, $N(x)=N_V(x)$ and $d(x)=d_V(x)$. For a graph $G$ and $X,Y \subseteq V(G)$ we use $ \nabla(X,Y)$ for the set of edges of $G$ joining $X$ and $Y$, and $E(X)$ for the set of edges contained in $X$.

We write $a= (1\pm \epsilon) b$ for $(1-\epsilon)b< a < (1+\epsilon)b$ and write $a\neq (1\pm \epsilon) b$ when this is not the case. 
We use $B(m, \alpha)$ for a random variable with the binomial distribution Bin$(m, \alpha)$ and $\log$ for $\ln$.

Set $$\varphi(x)=(1+x)\log (1+x)-x$$
for $x >-1$ and (for continuity) $\varphi(-1)=1$.
We will use the following form of Chernoff's inequality, which may be found, for example in \cite[Theorem 2.1]{JLR}

\begin{thm}\label{cher}
If $X \sim $ Bin$(m,q)$ and $\mu =\E[X]= mq$ then for $ \lambda \ge 0$ we have 
\begin{align} 
    \pr(X > \mu + \lambda) & \le \exp[-\mu \varphi (\lambda/\mu)]\le \exp \left[-\frac{\lambda^2}{2(\mu + \lambda/3)}\right] \label{cher1} \\
    \pr(X < \mu - \lambda) &\le \exp[-\mu \varphi (-\lambda/\mu)] \le \exp \left[- \frac{\lambda^2}{2\mu}\right]
\end{align}

\end{thm}
For larger deviations we use a consequence of the finer bound in (\ref{cher1}); see e.g. \cite[Theorem A.1.12]{AS}

\begin{thm} \label{ld}
For $X \sim B(m,q)$ with $\mu= \E[X]= mq$ and any $K$ we have 
\begin{align*}
\pr(X>Kmq) < \exp [-Kmq\log(K/e)].
\end{align*}
\end{thm}

For the rest of this section we set $G= G_{n,p}$ and assume that $p \ge \frac{8 \log n}{n}$. Some of the following statements hold in more generality, but this is not relevant for us. 

\begin{prop} \label{int}
For any $\epsilon>0$ for all $X \subseteq V(G)$ with $|X|>\epsilon n$ w.h.p.\
\begin{equation}\label{inteq1}
|E(X)| = (1 \pm \epsilon) {|X| \choose 2}p.
\end{equation}
Additionally, w.h.p.\ for all $X \subseteq V(G)$ with $|X| > \frac{\log n}{150 p}$ 
\begin{equation}\label{inteq2}
    |E(X)| \le 300 {|X| \choose 2}p.
\end{equation}

\end{prop}

\begin{proof}
For (\ref{inteq1}), where we may of course assume $\epsilon$ is fairly small, we first observe that for any $X \subseteq V(G)$ of size $w$, Theorem \ref{cher} gives (say)
$$\pr \left(|E(X)| \neq (1\pm\epsilon){w \choose 2}p\right) \le \exp\left[-\frac{\epsilon^2}{3}{w \choose 2}p\right].$$ 
So, the probability that there is some $X$ of size $w>\epsilon n$ violating (\ref{inteq1}) is no more than 
\begin{align*}
{n \choose w}\exp\left[-\frac{\epsilon^2}{2}{w \choose 2}p\right] &\le \exp\left[w \left(\log(en/w) - \frac{\epsilon^2}{4} \left(w-1\right)p\right)\right]\\
&<\exp[-(2-o(1))\epsilon^3w\log n],
\end{align*}
and summing over $w >\epsilon n$ bounds the overall probability that (\ref{inteq1}) fails by 
\begin{align*}
\sum_{w}\exp[-(2-o(1))\epsilon^3w\log n]&= o(1).
\end{align*}


For (\ref{inteq2}), fix $X \subseteq V(G)$ of size $w >\frac{ \log n}{150 p}$. Theorem \ref{ld} gives
\begin{align*}
\pr\left(|E(X)| > 300{w \choose 2}p\right) &\le \exp\left[-300\log(300/e){w \choose 2}p\right]\\
&<\exp[-700w(w-1)p].
\end{align*}

So, the probability that there is some $X$ of size $w>\frac{\log n}{150p}$ violating (\ref{inteq2}) is no more than
\begin{align*}
  {n \choose w}\exp[-700w(w-1)p] &\le \exp\left[-3w\log n\right] = n^{-3w},
\end{align*}
and summing over $w >\frac{\log n}{150p}$ we have $$\sum_w n^{-3 w}=o(1).$$
\end{proof}

\begin{prop} \label{bigint}
W.h.p.\ for all $X\subseteq V(G)$ with $ |X| \le \frac{\log n}{150p}$ 
\begin{equation}\label{bigint1}
    |E(X)| \le \frac{|X|\log n}{3}.
\end{equation}
\end{prop}

\begin{proof}
Note that if $|X|<\frac{2\log n}{3}$ then the statement is trivially true. Thus we now only consider $X\subseteq V(G)$ of size $w \in \left[\frac{2\log n}{3}, \frac{\log n}{150p}\right].$ On the other hand, for any such $X$ Theorem \ref{ld} gives 
\begin{align*}
\pr(|E(X)| \ge w \log n/3) &\le \exp\left[- (1/3)w\log n\log\left(\frac{2\log n}{3ewp}\right)\right]\\
&\le \exp\left[-(3/2)w\log n \right],
\end{align*}
where the final inequality holds since $w\le \frac{\log n}{150p}$.
So, the probability that there is some $X$ of size $w \in \left[\frac{2\log n}{3}, \frac{\log n}{150p}\right]$ violating (\ref{bigint1}) is no more than
\begin{equation*}
{n \choose w}\exp[-(3/2)w\log n]< n^{-w/2}, 
\end{equation*}
and summing over $w \in \left[\frac{2\log n}{3},\frac{\log n}{150p}\right]$  we have
$$\sum_w n^{-w/2} =o(1).$$
\end{proof}

\begin{prop} \label{between}
For any fixed $\epsilon>0$ w.h.p.\ 
\begin{equation}\label{betweeneq}
    |\nabla(Y,Z)| = (1\pm\epsilon) |Y||Z|p
\end{equation}
whenever $Y, Z \subseteq V(G)$ are disjoint and satisfy $|Y|>\epsilon n$ and 
and $|Z|>\frac{n}{\log^{1/2} n}$.
\end{prop}
There is nothing special about the value $\frac{n}{\log^{1/2} n}$; it is chosen to ensure that $|Z|=\omega(1/p)$ and $yp \gg \log(en/z)$. Additionally, the value $\frac{n}{\log^{1/2}n}$ will be a convenient cut-off later, so it is unnecessary to prove the lemma in more generality. 

\begin{proof}
We may of course assume $\epsilon$ is fairly small, and observe that for any given $Y,Z \subseteq V(G)$ disjoint with sizes $y$ and $z$, Theorem \ref{cher} gives
$$\pr(\nabla(Y,Z)\neq (1\pm \epsilon)yzp)\le \exp\left[-\frac{\epsilon^2}{3}yzp\right].$$
So, the probability that there are disjoint sets of sizes $y$ and $z$ violating (\ref{betweeneq}) is no more than
\begin{align*}
{n \choose y}{n \choose z}\exp\left[-\frac{\epsilon^2}{3}yzp\right] &\le\exp\left[y\log(en/y)+z\log(en/z)-\frac{\epsilon^2}{3}yzp\right]\\
&\le \exp\left[-\epsilon^2 y\log^{1/2} n\right]
\end{align*}
Summing over the appropriate values of $y$ and $z$ we have
\begin{equation*}
    \sum_y\sum_z\exp\left[-\epsilon^2 y\log^{1/2} n\right]=o(1).
\end{equation*}

\end{proof}

\begin{prop} \label{prop4}
W.h.p.\ for all fixed $a,b,c$ with 
\begin{itemize}
    \item $c-1\le b<a$,
    \item $c+b+a=n$,
    \item and $a>\frac{33n}{50}$
\end{itemize}
if $V(G)= A \cup B\cup C$ is a partition such that $|A|=a$, $|B|=b$, and $|C|=c$ then
$$|\nabla(A,B)| \ge .1 abp.$$
Additionally, w.h.p.\ if $b,c\ge \frac{n}{\log^{1/2} n}$ then
$$|\nabla(B,C)|\le 3bcp.$$
\end{prop}
\begin{proof}
First we note that the statement holds for all $B$ with, for example, $b<\log^{1/2}n$ because of a bound on the minimum degree. For example, given our bound on $p$, we know w.h.p.\ the minimum degree is, for example, at least $\frac{5np}{12}$. By Theorem \ref{cher} the probability that there is a vertex of smaller degree is at most
\begin{equation*}
    n\exp\left[-\frac{np}{6}\right]\le n^{-1/3}.
\end{equation*}
We know 
\begin{equation*}
|\nabla(A,B)| = \sum_{v\in B}d(v)-2|E(B)|-|\nabla(B,C)|\ge \frac{5|B|np}{12}-2\log n > .1|A||B|p.
\end{equation*}
For fixed $A,B$ such that $|A|=a>\frac{33n}{50}$ and $|B|=b\ge \log^{1/2}n$ Theorem \ref{cher} gives 
\begin{equation*}
    \Pr(|\nabla(A,B)|<.1abp)\le \exp\left[\frac{-.9^2abp}{2} \right].
\end{equation*}
Thus the probability that any such $A,B$ have $|\nabla(A,B)|<.1abp$ is at most
\begin{align*}
    {n \choose c}{n \choose b}\exp\left[-.4abp \right]&\le
    \exp[c\log n +b \log n-.4abp]\\
    &<\exp[(2b+1)\log n - 2.11b\log n]\\
    &< n^{-b/10}.
\end{align*}
Summing over $b>\log^{1/2}n$ and $c\le b+1$ we have
\begin{equation*}
    \sum_b \sum_{c}n^{-b/10}=o(1).
\end{equation*}

In the second case we know again by Theorem \ref{cher} that for fixed $B$ and $C$
\begin{equation*}
    \Pr(|\nabla(B,C)|> 3bcp)\le \exp\left[-1.2bcp \right].
\end{equation*}
Thus the probability that any such $B,C$ have $|\nabla(B,C)|>3bcp$ is at most
\begin{align*}
    {n \choose c}{n \choose b}\exp\left[-1.2bcp \right]&\le\exp\left[(2b+1)\log( e\log^{1/2}n)-9b \log^{1/2} n\right].
\end{align*}
Summing over $b,c \ge \frac{n}{log^{1/2}n}$ with $c\le b+1$ we have
\begin{equation*}
    \sum_b \sum_{c}\exp\left[(2b+1)\log( e\log^{1/2}n)-9b \log^{1/2} n\right]=o(1).
\end{equation*}
\end{proof}

\section{Proof of Theorem \ref{main}}\label{mainpf}
\begin{proof}[\unskip\nopunct]
In what follows the vertex set of each subgraph is $V$ (the vertex set of $G$). Thus we identify subgraphs of $G$ with their edge sets. We also abusively use simply ``component" for the vertex set of a component. We will show that w.h.p.\ for any subgraph of $G_{n,p}$ with matching number $k$ not of either form (\ref{form1}) or (\ref{form2}) we can construct a larger subgraph with the same matching number. To do this we rely on the Tutte-Berge formula (see e.g. \cite[Corollary 3.3.7]{W}):
\begin{thm}\label{TB}
For every graph $G =(V,E)$, $$|V|-2\nu(G)=\max_{S\subseteq V} o(G- S)-|S|,$$
where $o(G - S)$ is the number of odd components of $G- S$.
\end{thm}

Suppose $H$ is a largest subgraph of $G$ with matching number $k$. By Theorem \ref{TB} we may assume that there is a partition 
\begin{align}
V= S \cup A_1 \cup \cdots \cup A_d \label{config}
\end{align} with each $|A_i|$ odd and $d-|S|=n-2k$, such that $H$ consists of those edges of $G$ that are either incident to $S$ or contained in some $A_i$. For suppose $K \subseteq G$ has $\nu(K)=k$ and $S \subseteq V$ satisfies $n-2k=o(K - S)-|S|$ as in Theorem~\ref{TB}, with $B_1, \ldots B_d$ and $C_1, \ldots, C_l$ the odd and even components of $K - S$, respectively. Then letting $A_1 = B_1 \cup \bigcup_{i=1}^l C_i$ and $A_i=B_i$ for $i \ge 2$ we find that $H$ as above contains $K$ and still satisfies $\nu(H)=k$. We must only check that if $H$ is of either desired form then w.h.p.\ $|H|>|K|$. (If $H$ is not of the desired form then the trivial $|H|\ge |K|$ is enough, as we will show that for every subgraph with a partition as in (\ref{config}) not of the two desired forms we can w.h.p.\ construct a larger subgraph without increasing the matching number.)
First note that $H$ cannot be of form (\ref{form2}) since $|A_1|\ge 3$. 
If $H$ is of form (\ref{form1}) then
\begin{equation*}
    \left(\bigcup_{i=1}^d B_i \right)\cup \left(\bigcup_{i=1}^l C_i\right)=V.
\end{equation*}
Furthermore, we know that w.h.p.\ any partition of $G$ into two non-empty sets $B$ and $C$ has $|\nabla(B,C)|>0$. To see this fix such a partition. We may assume $|B|\ge n/2$, and we have
\begin{equation*}
    \Pr(|\nabla(B,C)|=0)=(1-p)^{|B||C|}\le \exp[-|B||C|p]\le \exp[-4|C|\log n],
\end{equation*}
where the last inequality uses $p\ge\frac{8\log n}{n}$ and $|B|\ge n/2$.
Taking the union bound over all choices for $C$ we get that the probability there is a partition with no crossing edges is at most
\begin{align*}
    \sum_{|C|\le n/2}{n \choose |C|}\exp[-4|C|\log n]&\le \sum_{|C|\le n/2}\exp[|C|(\log(en/|C|)-4\log n)]\\
    &\le \sum_{|C|\le n/2}n^{-2|C|}=o(1).
\end{align*}
Thus (with some reordering of the odd components) we may assume that $$\left|\nabla\left(B_1, \bigcup_{i=1}^l C_i\right)\right|>0$$ as desired. 

\medskip

For convenience we will always assume $|A_i| \ge |A_j|$ for all $i < j$. Now forms (\ref{form1}) and (\ref{form2}) correspond to configurations where:
\begin{enumerate}[label=(\alph*)]
\item $S= \emptyset$, $|A_1|= 2k+1$, and $A_2, \ldots A_d$ are single vertices; \label{config1}
\item $|S|=k$ and $A_1,\ldots A_d$ are single vertices. \label{config2}
\end{enumerate}
Since every subgraph we consider is specified by a vertex decomposition as in (\ref{config}), Theorem \ref{main} may be rewritten as a statement about such decompositions. The following notation will be helpful. We use $\Pi$ for decompositions as in (\ref{config}) (formally $\Pi$ is the partition of $V$ with blocks $S, A_1, \ldots, A_d$). We let $d(\Pi)$ to be the number of $A_i's$, $s(\Pi) = |S|$, and $r(\Pi)= d(\Pi)-s(\Pi)$. We say that the \emph{size} of $\Pi$, denoted $|\Pi|$, is the number of edges in the corresponding $H$. The EG Property for $G$ then becomes: 
\begin{center}
for each $k \le \nu(G)$ every largest $\Pi$ with $r(\Pi)= n-2k$ is of one of the forms \ref{config1}, \ref{config2}.
\end{center}
To prove the theorem in this form we show that for any given $k \le \nu(G)$ and decomposition $\Pi$ with $r(\Pi)= n-2k$ not of form \ref{config1} or \ref{config2} there is a larger $\Pi'$ with $r(\Pi)= r(\Pi')$. Since we will be comparing the sizes of two decompositions it is convenient to use the notation $S$ and $A_i$ for the blocks $\Pi$ while $S'$ and $A_i'$ are the blocks of $\Pi'$. Similarly we use $H$ for the edge set of $\Pi$, and $H'$ for the edge set of $\Pi'$. Additionally, it will be helpful to set 

$$B= \bigcup_{i=2}^d A_i\hspace{.5 cm}; \hspace{.5 cm}y(\Pi) = |B|-(d(\Pi)-1).$$

Note that $y(\Pi)$ is the number of ``excess'' vertices in $B$, and that in both desired forms $|B|=d-1$ (and thus $y(\Pi)=0$).
Since $d(\Pi)-s(\Pi)\ge 0$ we have for all decompositions $\Pi$ 
\begin{align}
\left|\bigcup_{i=1}^d A_i\right| \ge n/2. \label{Asize}
\end{align}

How we proceed will depend on the particulars of the given $\Pi$, which we divide into seven primary cases. For all but the last of these the argument is deterministic in the sense that we show the existence of the desired $\Pi'$ provided $G$ satisfies the conclusions of Propositions \ref{int}-\ref{prop4}; recall these were: \\ 
\underline{\ref{int}:} For any $\epsilon>0$ for all $X \subseteq V(G)$ with $|X|>\epsilon n$ w.h.p.\
\begin{equation*}
|E(X)| = (1 \pm \epsilon) {|X| \choose 2}p.
\end{equation*}
Additionally, w.h.p.\ for all $X \subseteq V(G)$ with $|X| > \frac{\log n}{150 p}$ 
\begin{equation*}
    |E(X)| \le 300 {|X| \choose 2}p.
\end{equation*}
\underline{\ref{bigint}:} W.h.p.\ for all $X\subseteq V(G)$ with $ |X| \le \frac{\log n}{150p}$ 
\begin{equation*}
    |E(X)| \le \frac{|X|\log n}{3}.
\end{equation*}
\underline{\ref{between}:} For any fixed $\epsilon>0$ w.h.p.\
\begin{equation*}
    |\nabla(Y,Z)| = (1\pm\epsilon) |Y||Z|p
\end{equation*}
whenever $Y, Z \subseteq V(G)$ are disjoint and satisfy $|Y|>\epsilon n$ and 
and $|Z|>\frac{n}{\log^{1/2} n}$.\\
\underline{\ref{prop4}}: W.h.p.\ for all fixed $a,b,c$ with 
\begin{itemize}
    \item $c-1\le b<a$,
    \item $c+b+a=n$,
    \item and $a>\frac{33n}{50}$
\end{itemize}
if $V(G)= A \cup B\cup C$ is a partition such that $|A|=a$, $|B|=b$, and $|C|=c$ then
$$|\nabla(A,B)| \ge .1 abp.$$
Additionally, w.h.p.\ if $b,c\ge \frac{n}{\log^{1/2} n}$ then
$$|\nabla(B,C)|\le 3bcp.$$
The cases are divided as follows:
\begin{enumerate}
    \item $|A_1|<n/2000$ and $y(\Pi)\ge n/2000$.
    \item $|A_1|,y(\Pi) <n/2000$.
    \item $|A_1|\ge n/2000$ and $|A_1|\le 3.99|B|$.
    \item $|A_1|>3.99 |B|$ and $y(\Pi)\ge 10^{-4}n$.
    \item $|A_1|>3.99 |B|$ and $0<y(\Pi)<10^{-4}n$.
    \item $|A_1|>3.99 |B|$, $0=y(\Pi)$, and $s(\Pi)>n/\log^{1/2} n$ or $|B|<\log^{1/2} n$.
    \item $|A_1|>3.99 |B|$, $0=y(\Pi)$, $|B|\ge \log^{1/2} n$, and $s(\Pi)<n/\log^{1/2} n$.
\end{enumerate}
These clearly cover all possible $\Pi$. (None of the inequalities are tight for their given argument, but are merely convenient cut-offs.) Again, in each of cases 1-6 we show Propositions \ref{int}-\ref{prop4} imply the existence of a $\Pi'$ that is larger than $\Pi$. We will say more about the final case when we come to it. \\

\subsection{Case 1: $|A_1|<n/2000$ and $y(\Pi)\ge n/2000$}  

Arbitrarily select $x_i \in A_i$ for $i \ge 2$, and let $M_i = A_i \setminus \{x_i\}$. To form $\Pi'$ let 
\begin{align}
A'_1&= A_1 \cup \bigcup_{i=2}^d M_i,\\
A'_i&= \{x_i\} \text{ for } i \ge 2, \text{ and }\\
S'&=S.
\end{align}
Note that since $d(\Pi)= d(\Pi')$ and $s(\Pi)=s(\Pi')$ we have $r(\Pi)=r(\Pi')$, as desired. 
We now check that $\Pi'$ is, in fact, larger than $\Pi$. First note that 
\begin{align}|H \setminus H'| = \sum_{i=2}^d d_{A_i}(x_i).\end{align}
Clearly $d_{A_i}(x_i) \le |A_i|-1$, so $|H \setminus H'| \le y <n$.
Furthermore, let $M_1 = A_1$. Then $H' \setminus H$ contains all edges joining distinct $M_i$'s, so for any $j$,
\begin{equation}\label{eq:1}
|H' \setminus H| \ge \left|\nabla \left(\bigcup_{i=1}^j M_i,\bigcup_{i=j+1}^n M_i\right)\right|.
\end{equation}

On the other hand if we take $j$ minimum with $|\bigcup_{i=1}^j M_i| =\epsilon n$ (for some $\epsilon>0$) then $y\ge n/2000$ implies $\bigcup_{i=j+1}^n M_i$ has size $\epsilon_1 n$. Thus by Proposition \ref{between} w.h.p.\ the cardinality in (\ref{eq:1}) is $\Omega(n\log n)$. Hence w.h.p.\ $|H \setminus H'| < |H' \setminus H|$ as desired.

\subsection{Case 2: $|A_1|,y(\Pi) <n/2000$} 

In this case, we select a particular $x$ such that $\{x\}= A_k$ for some $k$ and let $S'= S \cup \{x\}$. This clearly increases $s$ and decreases $d$, so we then select two vertices, $v$ and $z$, in some $A_j$ and let $\{v\}$ and $\{z\}$ be two new singleton $A'_i$'s, which maintains $r(\Pi)= r(\Pi')$. ($A_i = A'_i$ for all other $A_i$'s.) In order to ensure that $\Pi'$ is larger than $\Pi$ we must carefully select $x, v, $ and $z$ as follows.

Let $L$ be the set of singleton $A_i$'s and $d_1$ be the number of non-singleton $A_i$'s (where $i\ge 2$). Note that the number of vertices in non-singleton $A_i$'s is $y+d_1$. However, since each non-singleton component must contain at least $3$ vertices $d_1\le y/2$. Using this and (\ref{Asize}) we have $|L| \ge n/2-3y/2-|A_1|$. Since $|A_1|$ and $y$ are both at most $n/2000$ we (easily) have $|L| \ge n/3$. Hence, by Proposition \ref{int} we may assume that for any fixed $\epsilon_1>0$ $$|E(L)| \ge (1- \epsilon_1) {|L| \choose 2}p.$$ Thus there is a singleton $x$ such that $$\d_{\bigcup A_i}(x) \ge \d_L(x) \ge (1-\epsilon_1)(|L|-1)p.$$ Therefore, by letting $x$ be the singleton moved into $S'$ we have $$|H' \setminus H| = d_{\bigcup A_i}(x) \ge (1-\epsilon_1)(|L|-1)p.$$

To guarantee $\Pi'$ is larger than $\Pi$ we must ensure that w.h.p.\ we can select some $A_j$ and $v,z \in A_j$ such that $$\d_{A_j}(v)+\d_{A_j}(z) \le (1-\epsilon)(|L|-1)p,$$ since $|H \setminus H'|\le d_{A_j}(v)+d_{A_j}(z)$. 
If there is some $A_j$ with $|A_j|<\frac{\log n}{150p}$ then Proposition \ref{bigint} gives w.h.p.\ $$|E(A_j)|<\frac{|A_j|\log n}{3}.$$ Thus there are $v,z \in A_j$ such that $d_{A_j}(v)+d(A_j)(z)<2\log n$. Since 
\begin{equation*}
(1-\epsilon_1)(|L|-1)p>2\log n
\end{equation*}(for an appropriate choice of $\epsilon_1$) such $v,z$ suffice.
Finally if all $A_j$ have size at least $\frac{\log n}{150p}$ Proposition \ref{int} gives w.h.p.\
\begin{equation*}
    |E(A_j)|\le300{|A_j| \choose 2}p.
\end{equation*}
Thus we may assume there are $v,z\in A_j$ such that $d_{A_j}(v)+d_{A_j}(z)\le 600 |A_j|p$. However, recall for all $j$ we have $|A_j|<n/2000$, giving $$600|A_j|p<\frac{3np}{10}<(1-\epsilon_1)(|L|-1)p$$
(again for appropriate choice of $\epsilon_1$).

\subsection{Case 3: $|A_1|\ge n/2000$ and $|A_1|\le 3.99|B|$} 
Here we arbitrarily split $A_1$ into $A_1^1$ and $A_1^2$ where $|A_1^1|= \lfloor |A_1|/2 \rfloor$ and $|A_1^2|= \lceil |A_1|/2 \rceil$. We let $S'= S\cup A_1^1$, and let every vertex in $A_1^2$ become its own singleton $A'_i$. (All other $A_i$'s remain unchanged in $\Pi'$.) Note that $s(\Pi')= s(\Pi)+ \lfloor |A_1|/2 \rfloor$ and $d(\Pi')= d(\Pi)-1+ \lceil |A_1|/2 \rceil$. Since $|A_1|$ is odd $r(\Pi')= r(\Pi)$.

Again we must ensure that $\Pi'$ is larger than $\Pi$. First note that, $H \setminus H' = E(A_1^2)$. Since $|A_1^2|>n/4000$, for any fixed $\epsilon_3>0$, Proposition \ref{int} gives w.h.p.\ $$|E(A_1^2)| \le (1+ \epsilon_3){|A_1^2| \choose 2}p.$$ Additionally, $H' \setminus H = \nabla(A_1^1,B)$. Since we easily have $|A_1^1|,|B|\ge n/8000$, for any fixed $\epsilon_4>0$, Proposition \ref{between} gives w.h.p.\ $|\nabla(A_1^1,B)| \ge (1-\epsilon_4)|A_1^1||B|p$. Given our assumption that $|A_1| \le 3.99|B|$ it is simple to check that 
\begin{equation*}
\frac{(1+\epsilon_3) |A^2_1|^2 p}{2}<(1-\epsilon_4)|A_1^1||B|p
\end{equation*} for appropriate choices of $\epsilon_3$ and $\epsilon_4$. 

\subsection{Case 4: $|A_1|>3.99 |B|$ and $y(\Pi)\ge 10^{-4}n$} 
Here we create $\Pi'$ in the same manner as in Case 1 (moving all but one vertex of each $A_i$ for $i \ge 2$ to $A_1$). Again since $s(\Pi)= s(\Pi')$ and $d(\Pi)= d(\Pi')$ we still have $r(\Pi)=r(\Pi')$. As before, $|H \setminus H'| \le y$. Let $M$ be the set of vertices moved from $B$; then, $H' \setminus H \supseteq \nabla(A_1,M)$. Since $y\ge 10^{-4}n$ (and thus also, say, $|A_1|>10^{-5}n$) for any fixed $\epsilon_5>0$ Proposition \ref{between} gives w.h.p.\
\begin{equation*}
|\nabla(A_1,M)| \ge (1-\epsilon_5)|A_1|yp,
\end{equation*} which is larger than $y$ as desired.

\subsection{Case 5: $|A_1|>3.99 |B|$ and $0<y(\Pi)<10^{-4}n$} 

In this case to create $\Pi'$ we first select a particular $M \subseteq B$ with $|M|=y$ and let $A'_1 = A \cup M$. Then to keep $d(\Pi')= d(\Pi)$ we let all the vertices in $B \setminus M$ form their own singleton $A'_i$'s. Thus we maintain that 
\begin{equation*}
d(\Pi')= |B|- y +1 = d(\Pi).
\end{equation*}
Since we let $S'= S$ we have $r(\Pi')= r(\Pi)$. Again note $H' \setminus H \supseteq \nabla(A_1,M)$. 
To choose an appropriate $M$ first note that since $|A_1| > 3.99 |B|$ and $|B|\ge d+1> |S|$ we have 
$$n= |A_1|+|B|+|S| < \frac{50 |A_1|}{33}.$$
So, by Proposition \ref{prop4} we may assume $$|\nabla(A_1,B)| \ge .1 |A_1||B|p.$$ Thus there is some $M \subseteq B$ with $|M|=y$ such that 
\begin{equation}\label{case5Mbound}
|\nabla(A_1,M)| \ge .1|A_1|yp.
\end{equation}
Furthermore, $H \setminus H' \subseteq \bigcup_{i=2}^d E(A_i) \subseteq E(B_1)$, where $B_1$ is the set of all the vertices in $B$ not in a singleton $A_i$. 
By Proposition \ref{bigint} if $|B_1|\le \frac{\log n}{150p}$ we have w.h.p.\
\begin{equation}\label{b1int}
    |E(B_1)| <  \frac{|B_1| \log n}{3} .
\end{equation}
Since $|B_1|\le 3y/2$ we know (\ref{b1int}) is at most
$$ \frac{y \log n}{2}.$$
Combining this with (\ref{case5Mbound}) we have
\begin{equation*}
    |H' \setminus H|>.1|A_1|yp> \frac{y \log n}{2} >|H\setminus H'|,
\end{equation*}
where the second inequality follows easily from $|A_1|>\frac{33n}{50}$ and $p\ge \frac{8\log n}{n}$.
On the other hand if $|B_1|> \frac{\log n}{150p}$ then Proposition \ref{int} gives 
\begin{equation}\label{b1int2}
    |E(B_1)|<300{|B_1| \choose 2}p <350y^2p.
\end{equation}
Since $y<10^{-4}n$ and $|A_1| > \frac{33n}{50}$ we easily have (\ref{b1int2}) is less than $.1|A_1|yp$.

\subsection{Case 6: $|A_1|>3.99 |B|$, $0=y(\Pi)$, and $s(\Pi)>n/\log^{1/2} n$ or $|B|<\log^{1/2} n$} 
Since $y(\Pi)=0$ we know every $A_i$ for $i\ge 2$ is a singleton. 
To form $\Pi'$ we select $M \subseteq B$ of size $s(\Pi)$ and let $A'_1= A_1 \cup S \cup M$. Thus $S'= \emptyset$, and the $A'_i$ for $i \ge 2$ are simply those in $B \setminus M$.
Note that since both $s(\Pi')= 0$ and $d(\Pi')= d(\Pi)-s(\Pi)$ we have $r(\Pi')= d(\Pi)-s(\Pi)= r(\Pi)$. Here $H'\setminus H \supseteq \nabla(A_1, M)$. 

First consider when $ |S|>n\log^{-1/2}n$ (which also implies $|B|\ge n\log^{-1/2}n$).
For any fixed $\epsilon_6>0$ Proposition \ref{between} gives w.h.p.\
\begin{equation*}
|\nabla(A_1,M)|\ge (1-\epsilon_6)|A_1||M|p.
\end{equation*}
Similarly, $H \setminus H' = \nabla(S,B \setminus M)$, and by Proposition \ref{prop4} we may assume
\begin{equation*}
|\nabla(S,B \setminus M)| \le |\nabla(S,B)| < 3|S||B|p.
\end{equation*}
Since $|M|= |S|$ and $|A_1| \ge 3.99|B|$ we easily have $$(1-\epsilon_6)|A_1||M|p\ge 3|S||B|p$$ (for an appropriate choice of $\epsilon_6$).

If $|B|<\log^{1/2} n$ Proposition \ref{prop4} allows us to assume $$|\nabla(A_1,B)| \ge .1 |A_1||B|p.$$ Given this we can select $M$ with $$|\nabla(A_1,M)|> .1 |A_1||M|p =.1|A_1||S|p.$$ However, trivially, $$|\nabla(S,B \setminus M)|<\log^{1/2} n |S|$$ Thus we have \begin{equation*}
    |H' \setminus H|> .1|A_1||S|p > \log^{1/2} n |S| >|H \setminus H'| 
\end{equation*}as desired. 

\subsection{Case 7: $|A_1|>3.99 |B|$, $0=y(\Pi)$, $|B|\ge \log^{1/2} n$, and $s(\Pi)<n/\log^{1/2} n$}
In this case we will show that there is a partition $\Pi'$ larger than $\Pi$ that is of form \ref{config1} ($S'= \emptyset$, $|A'_1|=2k+1$, and $A'_2, \ldots, A'_d$ are all single vertices). For reference we restate that here we assume $\Pi$ has the following form:
\begin{enumerate}[label=(\alph*),start=3]
\item $0<s(\Pi)<n\log^{-1/2}n$, $A_1$ is the only non-singleton component, and $|A_1| > 3.99|B|$. \label{config3}
\end{enumerate}

\mn

In what follows, thinking of $\Pi$ as in (c), we will use $a$ for the size of $A_1$
and $b$ for $d(\Pi)-1$.  Let us note to begin that, since
\[
a = n-s-b = 2k-2s+1,
\]
any two of $k,s,a,$ and $b$ determine the others.
We assume throughout that whichever parameters we specify determine an $s$ and $a$ as in (c).

\mn

The present case differs from those above in that we need to be more conservative without use of the union bound. We can no longer afford to sum over possibilities for $B$. To avoid this we largely ignore the initial $\Pi$ and focus on $s=s(\Pi)$.

\mn
Precisely, we show that w.h.p.\ for every $k$, $s$, and $S$ of size $s$, the largest
$\Pi$ as in (c) with this $S$ and $r(\Pi)=n-2k$ is smaller than some $\Pi'$ as
in (a)
with
\begin{equation}\label{rpi}
r(\Pi')=r(\Pi).
\end{equation}

In analyzing what happens here we will use direct applications of Theorem \ref{cher} and Theorem \ref{ld}
(so in this case the argument is not ``purely deterministic").
Note that here, unlike in our earlier cases, simply assuming the ``w.h.p." statements of Section \ref{prelim}
causes trouble since further analysis then involves \emph{conditioning} on these properties,
and the resulting probability distribution is not one we are likely to understand.

Instead we identify, for each $k$ and $S$, a set of ``bad" events, say $E_k(S)$,
for which we can show, first, that
\begin{equation}\label{sumdS}
\sum_k\sum_S\pr(E_k(S)) =o(1)
\end{equation}
and, second, that if $E_k(S)$ does not occur, then there does exist some $\Pi'$ as above.
(Thus our union bound sums over choices of $k$ and $S$, but not $B$.)

\mn

To specify $\Pi$ for given $k,S$, we think of choosing the edges of $G-S$ and then those meeting $S$.  Given the first choice we may choose $A_1$ to be some $a$-subset of $V\sm S$ maximizing $|G[A_1]|$.  
(In case of ties we may, for example, assume some fixed ordering of the $a$-subsets of $V$ and take $A_1$ to be the first such maximizer in this ordering.)
Notice that, since the contribution to $|\Pi|$ of edges meeting $S$ doesn't depend on $A_1$, the $\Pi$ determined by this choice of $A_1$ is optimal for the given $k$ and $S$.

Given $A_1$ (equivalently, $\Pi$), we set $B= V\sm (S\cup A_1)$ (so $b=|B|$). To form $\Pi'$ from $\Pi$ we select $M \subseteq B$ with $|M|=s$, and let $A_1' = A_1 \cup M \cup S$.
Loosely, our bad events are:
 \begin{enumerate}
     \item $|\nabla(A_1,B)|$ is too small;
     \item $|\nabla(S,B)|$ is too large. 
 \end{enumerate}
Note that clearly for the first bad event we will need to take a union bound over all choices for $B$, but it is in the second bad event where we are able to avoid this.

The precise quantification will depend on $b$ (specifically, on whether it is $\Omega(n)$ or smaller). 

First assume $b>10^{-3}n$. Here (with .9 chosen for convenience) our bad events are:
 \begin{enumerate}
     \item $|\nabla(A_1,B)| < .9abp$;
     \item $|\nabla(S,B)| >.9asp$.
 \end{enumerate}

By Theorem \ref{cher} 
\begin{align*}
\pr(|\nabla(A_1,B)| \le .9abp) &\le \exp\left[- \frac{.1^2 abp}{2}\right].
\end{align*}
Therefore, 
\begin{align}
    \sum_k \sum_s \sum_{S: |S|=s} \sum_{B: |B|=b} \pr(|\nabla(A_1,B)| < .9abp) \le \nonumber\\
    \sum_k \sum_s \exp\left[s \log (en/s)+ b\log(en/b) - \frac{.1^2 abp}{2}\right]\le \nonumber \\
    \sum_k\sum_s \exp[-10^{-3}abp]. \label{e1}
\end{align}
Since $b>10^{-3}n$, $a>\frac{33n}{50}$, and $p\ge\frac{8 \log n}{n}$ we know (\ref{e1}) is, for example, at most
\begin{equation*}
    \sum_k\sum_s n^{-10^{-6}n} =o(1).
\end{equation*}

Additionally, for a fixed $X$ with $|X|=s$ and $Y$ with $|Y|=b$ Theorem \ref{cher} gives
\begin{align*}
  \pr(|\nabla(X, Y) |> .9asp) &\le \exp\left[-\frac{sp(.9a-b)^2}{2(b+(.9a-b)/3)}\right]. 
\end{align*}
since $a>3.99b$ one can check that $\frac{(.9a-b)^2}{2(b+(.9a-b)/3)}>.4a$.
Using this, $a>\frac{33n}{50}$, and $p\ge\frac{8\log n}{n}$ we have
\begin{align*}
\sum_k \sum_s \sum_{S:|S|=s} \pr(|\nabla(X, Y)| > .9asp) &\le \sum_k \sum_s \exp\left [s \log(e n/s)- .4asp \right] \\&\le \sum_k \sum_s n^{-1.1s}=o(1).
\end{align*}

Thus we know w.h.p.\ 
we can find some $M \subseteq B$ with $|M|=s$ such that $|\nabla(A_1, M)|>|\nabla(S,B)|$, ensuring that $\Pi'$ is larger than $\Pi$.

Now we assume $b<10^{-3}n$. Here (again with .1 chosen for convenience) our bad events are: 
 \begin{enumerate}
     \item $|\nabla(A_1,B)| < .1abp$;
     \item $|\nabla(S,B)| >.1asp$.
 \end{enumerate}
By Theorem \ref{cher} we have 
\begin{equation*}
    \pr(\nabla(A_1, B)<.1abp)< \exp\left[\frac{-.9^2abp}{2}\right].
\end{equation*}
Thus $p\ge \frac{8 \log n}{n}$ and $a\ge(1-2\cdot 10^{-3})n$ gives:
\begin{align*}
\sum_k\sum_s \sum_{S: |S|=s} \sum_{|B|=b} \pr(|\nabla(A_1,B)| < .1 abp) & \le \exp\left[(s+b)\log(en) - \frac{.9^2abp}{2}\right]\\
&\le \exp\left[(2b+1)\log(e n) - \frac{.9^2 abp}{2} \right]\\
&\le \sum_k \sum_s n^{-b}=o(1).
\end{align*}

Additionally, for a fixed $X$ and $Y$ with $|X|=s$ and $|Y|=b$ Theorem \ref{ld} gives
\begin{align*}
    \pr(|\nabla(X, Y) |> .1asp) &\le  \exp \left[ -.1asp \log \left[ \frac{ a}{10eb}\right]\right]
\end{align*}
Note that since $b<10^{-3}n$ it is easy to check that $$\log\left( \frac{a}{10eb}\right)>3.6.$$
Therefore we have
\begin{align*}
    \sum_k \sum_s \sum_{S:|S|=s} \pr(|\nabla(X, Y)| >.1 asp)& < \sum_k \sum_s \exp \left[s \left[\log en -.36 ap \right]\right] \\
    &<\sum_k\sum_s n^{-1.5s}=o(1).
\end{align*}

Hence, we again have w.h.p.\ that our bad events do not occur. Thus, we can again find some $M \subseteq B$ with $|M|=s$ such that $|\nabla(A_1, M)|>|\nabla(S,B)|$, ensuring that $\Pi'$ is larger than $\Pi$.

\end{proof}
\section{Conclusion}\label{final}

We first note a second regime where $G_{n,p}$ has the EG Property, which immediately follows from the two theorems below (see e.g \cite[Theorem 3.1.16]{W} and \cite{JLR}). Recall that $\tau(G)$ is the (vertex) cover number. 

\begin{thm}\label{Konig}(K\H{o}nig's Theorem)
If $G$ is a bipartite graph then $\nu(G)= \tau(G)$. 

\end{thm}

\begin{thm}\label{for}
If $p \ll 1/n$ then w.h.p. $G_{n,p}$ is a forest. 
\end{thm}

\begin{cor}
If $p \ll 1/n$ then w.h.p.  $G_{n,p}$ has the EG Property.
\end{cor}

\begin{proof}
Assume $p \ll 1/n$. By Theorem \ref{for} we know w.h.p. $G_{n,p}$ is a forest. Assuming $G_{n,p}$ is a forest we know by Theorem \ref{Konig} $\nu(H)= \tau(H)$ for all subgraphs $H$ of $G_{n,p}$. Thus for a given $k$ every largest subgraph of $G_{n,p}$ with matching number $k$ is the set of edges incident to a set of $k$ vertices. Hence $G$ has the EG Property. 
\end{proof}

We conclude by noting one regime where w.h.p. $G_{n,p}$ does not have the EG Property.

\begin{thm}\label{fails}
If $\frac{4\log(2e)}{n} < p< \frac{\log n}{3n}$, then w.h.p.\ $G_{n,p}$ does not have the EG Property.
\end{thm}

This is based on the following preliminaries:

\begin{prop}\label{P3}
For $p$ as in Theorem \ref{fails} w.h.p. $G_{n,p}$ contains at least two isolated $P_3$'s ($P_3$ is a path on 3 vertices).  
\end{prop}

Proposition $\ref{P3}$ is proved using a basic second moment method argument. 

\begin{thm}\label{CI}(Chebyshev's Inequality (see e.g. \cite[Theorem 4.1.1]{AS}))
Let $X$ be a random variable with expectation $\mu$ and standard deviation $\sigma$. For any $\lambda>0$ $$\pr(|X - \mu|\ge \lambda \sigma) < \frac{1}{\lambda^2}.$$

\end{thm}

\begin{proof}[Proof of Proposition \ref{P3}]
Let $p$ be as in Proposition \ref{P3} and let $X$ be the number of isolated $P_3$'s in $G$. We have $$\E X = 3 {n \choose 3} p^2 (1-p)^{3n-8}.$$ Note that for our values of $p$ we have $\E X \rightarrow \infty$. Furthermore, 
\begin{equation*}
\E X^2 = \E X + 9 {n \choose 3} {n-3 \choose 3}p^4 (1-p)^{6n-25}.
\end{equation*} This is because if $X_i$ and $X_j$ are both indicators of isolated $P_3$'s then $X_iX_j$ is always zero if there are some shared vertices and the paths are not identical. 
Thus it is easy to check that 
\begin{align*}
\lim_{n \rightarrow \infty} \frac{\E X^2}{\E^2 X} &= 1
\end{align*}
Therefore, Chebyshev's inequality easily gives the desired result.


\end{proof}

\begin{prop}\label{nonempty}
For $p$ as in Theorem \ref{fails} w.h.p. for all $X \subseteq V(G_{n,p})$ with $|X|= n/2$ we have $E(X)$ is non-empty.

\end{prop}

\begin{proof}
For any given $X \subseteq V(G_{n,p})$ we have 
\begin{align*}
\pr(|E(X)|=0)&= (1-p)^{|X| \choose 2} \le \exp\left[-p {|X| \choose 2}\right].
\end{align*}
Therefore, the probability that any vertex set of size $|X|=n/2$ has no edges is at most 
\begin{align*}
   {n \choose n/2} \exp\left[-p {n/2 \choose 2}\right] &\le \exp\left[n/2  (\log(2e)- p (n/2-1)/2)\right] =o(1),
\end{align*}
where $p > {4\log(2e)}/n$ gives the final equality.

\end{proof}

\begin{proof}[Proof of Theorem \ref{fails}]
We show w.h.p.\ the EG Property fails when $k = \nu(G_{n,p})$. Let $G=G_{n,p}$. Assuming the conclusions of Propositions \ref{P3} and \ref{nonempty} the remaining argument is deterministic. Clearly the largest subgraph with matching number $k$ is $G$ itself. Thus having the EG Property at $k$ is equivalent to satisfying one of the following two properties
\begin{enumerate}[label=(\alph*)]
\item All edges of $G$ are within a set of vertices of size $2k+1 $ ;\label{a}
\item $\tau(G)= \nu(G)$ .\label{b}
\end{enumerate}
Assuming the conclusion of Proposition \ref{P3} there are two isolated $P_3$'s, say $P'$ and $P''$, in $G$. Note $\nu(G - \{P',P''\})= k-2$. Thus if $X$ is the minimum set of vertices such that all edges of $G - \{P', P''\}$ are contained in $X$ we have $|X| \ge 2(k-2)$. However, to include $P'$ and $P''$ we need 6 more vertices. Thus we need at least $2k+2$ vertices to ensure that every edge in $G$ is included, violating case \ref{a}.\\
Furthermore, by Proposition \ref{nonempty} we have that every set of vertices of size $n/2$ has at least one edge. Thus $\tau(G)>n/2 \ge \nu(G)$, violating case \ref{b}.
\end{proof}

\printbibliography
\end{document}